\documentclass[reqno, 10pt]{amsart}
\usepackage{amssymb}
\usepackage{amsmath}
\usepackage[mathscr]{euscript}
\usepackage{hyperref}
\usepackage{graphicx}
\usepackage[normalem]{ulem}

\usepackage{amssymb}
\usepackage{hyperref}
\RequirePackage[dvipsnames]{xcolor} 
\definecolor{halfgray}{gray}{0.55} 
\definecolor{webgreen}{rgb}{0,0.5,0}
\definecolor{webbrown}{rgb}{.6,0,0} \hypersetup{%
  colorlinks=true, linktocpage=true, pdfstartpage=3,
  pdfstartview=FitV,%
  breaklinks=true, pdfpagemode=UseNone, pageanchor=true,
  pdfpagemode=UseOutlines,%
  plainpages=false, bookmarksnumbered, bookmarksopen=true,
  bookmarksopenlevel=1,%
  hypertexnames=true,
  pdfhighlight=/O,
  urlcolor=webbrown, linkcolor=RoyalBlue,
  citecolor=webgreen, 
  pdftitle={},%
  pdfauthor={},%
  pdfsubject={2000 MAthematical Subject Classification: Primary:},%
  pdfkeywords={},%
  pdfcreator={pdfLaTeX},%
  pdfproducer={LaTeX with hyperref}%
}

\newtheorem{theorem}{Theorem}

\theoremstyle{definition}

\newtheorem{example}{Example}[section]
\newtheorem{remark}{Remark}[section]

\renewcommand{\epsilon}{\varepsilon}

\def\Id{\text{\rm Id}}
\def\cA{\EuScript{A}}
\def\N{\mathbb{N}}
\def\Z{\mathbb{Z}}
\def\R{\mathbb{R}}

\begin{document}
\title[Multiscale linearization of nonautonomous systems]{Multiscale linearization of nonautonomous systems}

\author{Lucas Backes}
\address{\noindent Departamento de Matem\'atica, Universidade Federal do Rio Grande do Sul, Av. Bento Gon\c{c}alves 9500, CEP 91509-900, Porto Alegre, RS, Brazil.}
\email{lucas.backes@ufrgs.br} 
\author{Davor Dragi\v cevi\' c}
\address{Faculty of Mathematics, University of Rijeka, Croatia}
\email{ddragicevic@math.uniri.hr}

\begin{abstract}
We present sufficient conditions under which a given linear nonautonomous system and its 
 nonlinear perturbation  are topologically conjugated. Our conditions are of a very general form and provided that the nonlinear perturbations are well-behaved, we do not assume any 
  asymptotic behaviour of the linear system. Moreover, the control on the nonlinear perturbations may differ along finitely many mutually complementary directions. We consider both the cases of one-sided  discrete and continuous dynamics.
\end{abstract}

\maketitle

\section{Introduction}

Let us consider a (one-sided) linear nonautonomous difference equation given by
\begin{equation}\label{eq: linear system}
x_{n+1}=A_nx_n, \quad n\in \N,
\end{equation}
as well as its nonlinear perturbation
\begin{equation}\label{eq: nonlinear system}
x_{n+1}=A_nx_n + f_n(x_n), \quad n\in \N,
\end{equation}
where $A_n:X\to X$, $n\in \N$ is a sequence of invertible bounded linear operators acting on a Banach space $X$ and 
$f_n: X\to X$, $n\in \N$ is a sequence of  (nonlinear) maps. In this note we are interested in describing sufficient conditions under which the systems \eqref{eq: linear system} and \eqref{eq: nonlinear system} are \emph{topologically conjugated}, meaning that there exists a sequence of homeomophisms $H_n\colon X\to X$, $n\in \N$ mapping the trajectories  of~\eqref{eq: linear system} into trajectories  of~\eqref{eq: nonlinear system}.
 Whenever such conjugacies exist, many important dynamical properties of the \emph{nonlinear system} \eqref{eq: nonlinear system} can be obtained by studying the \emph{linear system} \eqref{eq: linear system}, which in general is much easier to deal with.

\emph{Linearization problems}, as the one described above, have a long history. As cornerstones of this theory (dealing with the case of autonomous dynamics),  we refer to the works of Grobman \cite{G1,G2} and Hartman \cite{H,H1,H2}.  The first linearization results dealing with the case of infinite-dimensional dynamics are due to Palis~\cite{Palis} and Pugh~\cite{Pugh}. The problem of formulating sufficient conditions under which the conjugacy exhibits higher regularity properties was first considered in the pioneering works of Sternberg~\cite{Stern, Stern2}.

The first nonautonomous version of the Grobman-Hartman theorem was established by Palmer~\cite{Palmer} for the case of continuous time. The discrete time version of his result (formulated in~\cite{AW}) asserts that~\eqref{eq: linear system} and~\eqref{eq: nonlinear system}
are topologically conjugated provided that the following conditions hold:
\begin{itemize}
\item \eqref{eq: linear system} admits an exponential dichotomy (see~\cite{Coppel});
\item the nonlinear terms $f_n$ are bounded and uniformly Lipschitz with a sufficiently small Lipschitz constant.
\end{itemize}
In addition, several authors obtained important extensions of the Palmer's theorem by relaxing  assumptions related to the linear systems~\eqref{eq: linear system} (or its continuous counterpart). We refer to~\cite{BDcol, BDP, CGR, CR, Jiang, Lin, RS} and references therein. 
For recent results dealing with the higher regularity of conjugacies in nonautonomous linearization, we refer to~\cite{BD2, CJ, CMR, CR0, CDS, D, DZZ, DZZ2}.

An ubiquitous assumption in most of  those  results  is the existence of a decomposition of the phase space $X$ into the stable and unstable directions along which~\eqref{eq: linear system} exhibits contraction and expansion, respectively. 
In other words, it is assumed that~\eqref{eq: linear system} exhibits some sort of dichotomic behaviour (although not necessarily of exponential nature). The key idea is that the lack of hyperbolicity can be compensated by properly controlling the ``size'' of nonlinear terms $f_n$ in~\eqref{eq: nonlinear system}. A notable exception is the work of Reinfelds and \v{S}teinberga~\cite{RS} in which the authors obtain a linearization result without any assumptions related to the asymptotic behaviour of~\eqref{eq: linear system}. However, the conditions concerned with nonlinearities $f_n$ are expressed in terms of a Green function corresponding to~\eqref{eq: linear system} which is still essentially given by decomposing $X$ into two directions.

In this work, instead of considering a decomposition of $X$ into just two directions, we allow for a decomposition of $X$ into several directions with possible different behaviours of \eqref{eq: linear system} along  each of those directions. Our conditions are of general form as in \cite{BDP,RS} and no asymptotic behaviour is required for the linear dynamics. On the other hand, the more ``non-hyperbolic" the linear system is (along certain direction),  the more restrictive are the assumptions on the perturbations $f_n$ (along that direction). In fact, we allow the presence of certain  directions along which we do not impose any conditions on~\eqref{eq: linear system} and on the  nonlinear perturbations $f_n$ (besides requiring those  to be continuous and bounded). In this general case, we obtain only a \emph{quasi-conjugacy} between systems~\eqref{eq: linear system} and~\eqref{eq: nonlinear system},  meaning that they are conjugated except for a given deviation in the directions along which we have no control. To the best of our knowledge, this is the first time such a general result appears in the literature. We stress that our results are motivated by the recent paper by Pilyugin~\cite{Pil} which deals with the multiscale nonautonomous  shadowing.

The paper is organized as follows. In Section~\ref{DT}, we consider the case of discrete time, i.e. we establish sufficient conditions under which~\eqref{eq: linear system} and~\eqref{eq: nonlinear system} are topologically conjugated. We discuss in detail the relationship between our result and related results in the literature and we provide an explicit example illustrating the strenght of our result.  Finally, in Section~\ref{CT} we establish an analogous result in the case of continuous time.

 \section{The case of discrete time}\label{DT}
\subsection{Preliminaries}\label{P}
Let $X=(X, | \cdot |)$ be an arbitrary Banach space and denote by $\mathcal B(X)$ the space of all bounded linear operators on $X$. By $\|\cdot\|$,  we will denote the operator norm on $\mathcal{B}(X)$. Given a sequence $(A_n)_{n\in \N}$ of invertible  operators in $\mathcal B(X)$ and $m,n\in \N$, let us consider the associated linear cocycle given by 
\[
\cA(m, n)=\begin{cases}
A_{m-1}\cdots A_n & \text{if $m>n$;}\\
\Id & \text{if $m=n$;} \\
A_m^{-1}\cdots A_{n-1}^{-1} & \text{if $m<n$.}
\end{cases}
\]

\subsection{Multiscale}\label{Multi} Let $K$ be a finite set of the form $K=K^s \cup K^u\cup K^c$, where  $K^i\cap K^j=\emptyset$ for $i,j\in \{s,u,c\}$, $i\neq j$. 
Suppose that for each $n\in \N$ there exists a family of projections $P^k_n$, $k\in K$  such that 
\begin{equation}\label{split}
\sum_{k\in K}P^k_n=\Id,
\end{equation}
and 
\begin{displaymath}
P^k_nP^l_n=0 \  \text{ for every } k,l\in K, \ k\neq l.
\end{displaymath}
In particular, by considering $X_k(n)=P^k_n(X)$, we have that
\begin{equation*}
X=\bigoplus_{k\in K} X_k(n) \quad \text{for every $n\in \N$.}
\end{equation*}

\begin{remark}
We observe that the notion of multiscale considered in this work is more general than the one considered by Pilyugin in \cite{Pil}. Indeed, in the aforementioned work there is an extra assumption requiring that $A_nP_n^k=P^k_{n+1}A_n$ for every $n\in \N$ and $k\in K$. Moreover,  in \cite{Pil} $K$ has the form $K=K^s\cup K^u$, i.e. $K^c=\emptyset$. 
\end{remark}

\subsection{Standing assumptions} Given $k\in K^s\cup K^u$, take $\lambda_k>0$ and let $(\mu^k_n)_{n\in \N}$ and $(\nu_n^k)_{n\in \N}$ be sequences of positive numbers such that:
\begin{itemize}
\item for $k\in K^s$,
\begin{equation}\label{eq: multi c1}
\sup_n \sum_{l=1}^n\|\cA(n,l)P^k_l\|\nu^k_l<+\infty, 
\end{equation}
and
\begin{equation}\label{eq: multi c2}
\sup_n \sum_{l=1}^n\|\cA(n,l)P^k_l\|\mu^k_l\leq \lambda_k;
\end{equation}
\item for $k\in K^u$,
\begin{equation}\label{eq: multi c3}
\sup_n \sum_{l=n+1}^{\infty}\|\cA(n,l)P^k_l\|\nu^k_l<+\infty, 
\end{equation}
and
\begin{equation}\label{eq: multi c4}
\sup_n \sum_{l=n+1}^{\infty}\|\cA(n,l)P^k_l\|\mu^k_l\leq \lambda_k.
\end{equation}
\end{itemize}

\subsection{A linearization result} 
We are now ready to state our first main result.

\begin{theorem}\label{theo: lienarization}
Let $f_n \colon X\to X$, $n\in \N$ be a sequence of maps such that $A_n+f_n$ is a homeomorphism for each $n\in \N$ and
\begin{equation}\label{bound of f}
\lVert P^k_n f_{n-1}\rVert_\infty \leq \nu^k_n, 
\end{equation}
for every $k\in K^s\cup K^u$ and $n\geq 1$, where
\[
\lVert P^k_n f_{n-1}\rVert_\infty:=\sup \{ |P^k_n f_{n-1} (x)|: x\in X\}.
\]
Moreover, assume that for each $k\in K^s\cup K^u$, $x, y\in X$ and $n\in \N$,
\begin{equation}\label{Lipschitz}
|P^k_nf_{n-1}(x)-P^k_nf_{n-1}(y)| \le \mu^k_n |x-y|.
\end{equation}
Then, if 
\begin{equation}\label{contr}
\sum_{k\in K^s\cup K^u}\lambda_k <1,
\end{equation}
\begin{itemize}
\item[i)] there exist sequences of continuous maps $H_n:X\to X$, $n\in \N$ and $\tau_n:X\to \bigoplus_{k\in K^c}X_k(n+1)$, $n\in \N$ such that 
\begin{equation}\label{lin1}
H_{n+1}\circ A_n=(A_n+f_n)\circ H_n+\tau_n\circ H_{n}, \ \text{for every $n\in \N$}.
\end{equation}
In addition, 
\[
\sup_{n\in \N}\lVert H_n-\Id \rVert_\infty <+\infty \text{ and } \tau_n(x)=-\sum_{k\in K^c}P^k_{n+1} (f_{n}(x));
\]

\item[ii)] there exist sequences of continuous maps $\bar H_n:X\to X$, $n\in \N$ and $\bar \tau_n:X\to \bigoplus_{k\in K^c}X_k(n+1)$, $n\in \N$ such that   
\begin{equation}\label{lin2}
\bar{H}_{n+1}\circ (A_n+f_n)=A_n \circ \bar{H}_n +\bar{\tau}_n, \  \text{for every $n\in \N$}.
\end{equation}
In addition, 
\[
\sup_{n\in \N}\lVert \bar H_n-\Id \rVert_\infty <+\infty \text{ and } \bar \tau_n(x)= -\tau_n.
\]

\end{itemize}
Moreover, in the case when either $K^c=\emptyset$ or $P^k_nf_{n-1}\equiv 0$ for every $k\in K^c$ and $n\geq 1$,  we have that $H_n$ and $\bar H_n$ are homeomorphisms for each $n\in \N$. In addition,
\begin{equation}\label{eq: H inverse theo1}
H_n\circ \bar H_n=\bar H_n\circ H_n=\Id 
\end{equation}
and 
\begin{equation}\label{eq: linearization theo1}
H_{n+1}\circ A_n=(A_n+f_n)\circ H_n,
\end{equation}
for every $n\in \N$.
\end{theorem}

\begin{remark}
We observe that in the case when we have a good control of the perturbations along all the directions (i.e., $K^c=\emptyset$ or $P^k_nf_{n-1}\equiv 0$ for every $k\in K^c$ and $n\geq 1$), the previous result gives us a nonautonomous version of Grobman-Hartman's theorem. In the general case, however, we obtain  a ``quasi-conjugacy" between systems \eqref{eq: linear system} and \eqref{eq: nonlinear system}, i.e. those are conjugated except for a given deviation (the factors $\tau_n$ and $\bar \tau_n$  in \eqref{lin1} and \eqref{lin2}, respectively) in the directions along which we do not have any control.
\end{remark}

\begin{remark}
Another important observation is the generality of Theorem \ref{theo: lienarization}: we do not impose any condition on the linear maps $(A_n)_{n\in \N}$ but rather  only on the allowed perturbations. Moreover, we allow for different levels of control on the perturbations along  different directions. In particular, it generalizes previous results such as \cite[Theorem 2.1]{CGR}.
\end{remark}

\begin{remark}\label{remm}
We note that the classical Palmer's theorem~\cite{AW, Palmer} corresponds to the particular case when:
\begin{itemize}
\item  $|K^s|=|K^u|=1$ and $K^c=\emptyset$;
\item there exist $D, \lambda >0$ such that \[ \| \cA(m, n)P_n\| \le De^{-\lambda (m-n)} \quad \text{for $m\ge n$},
\]
 and \[ \| \cA(m,n)\tilde P_n\| \le De^{-\lambda (n-m)} \quad  \text{for $m\le n$,} \]
where $P_n=P_n^a$ and $\tilde P_n=P_n^b$ for $n\in \N$, $K^s=\{a\}$, $K^u=\{b\}$.
\end{itemize}
It is easy to verify that in this setting Theorem~\ref{theo: lienarization} is applicable whenever $(\nu_l^k)$ and $(\mu_l^k)$ are constant sequences, and $\mu_l^k$ is sufficiently small.
\end{remark}

\begin{remark}
Let $K^s$ and $K^u$ satisfy the same properties as in Remark~\ref{remm}. 
We emphasize that in the case when $|K^c|=1$, 
a result similar to Theorem~\ref{theo: lienarization} was established (by using different techniques and under some additional assumptions) in~\cite[Theorem 3]{BD}. 
\end{remark}

As an illustration of the broad applicability of Theorem \ref{theo: lienarization} we provide the following simple example.
\begin{example}\label{EXM}
Take $X=\R^5$ and $K=\{1,2,3,4,5\}$. For each $k\in K$ and $n\in \N$, let $P^k_n$ be the projection onto the $k^{th}$ coordinate. Moreover, let $(A_n)_{n\in\N}$ be a sequence of constant diagonal matrices given by
\[
A_n=\text{diag}\left(\frac{1}{2},1,1,1,2\right) \text{ for every } n\in \N,
\]
and consider $K^s=\{1,2\}$, $K^c=\{3\}$ and $K^u=\{4,5\}$. Take $\lambda_k=\frac{1}{5}$ for every $k\in K$ and
\begin{itemize}
\item $\nu^k_n=1$ and $\mu^k_n=\frac{1}{10}$, for $k\in \{1,5\}$ and  $n\in \N$;
\item $\nu^k_n=\frac{1}{2^{n}}$ and $\mu^k_n=\frac{1}{5\cdot 2^{n}}$, for $k\in \{2,4\}$ and $n\in \N$.
\end{itemize}
Let $f_n\colon X\to X$, $f_n=(f_n^1, \ldots, f_n^5)$, $n\in \N$ be a sequence of continuous maps such that 
\begin{itemize}
\item $\|f_{n-1}^k\|_\infty \le 1$ and $Lip(f_{n-1}^k)\le \frac{1}{10}$, for $k\in \{1,5\}$ and $n\ge 1$;
\item $\|f_{n-1}^k\|_\infty \le  \frac{1}{2^{n}}$ and $Lip(f_{n-1}^k)\le \frac{1}{5\cdot 2^{n}}$, for $k\in \{2, 4\}$ and $n\ge 1$.
\end{itemize}
It is easy to verify that under the above assumptions,  Theorem \ref{theo: lienarization} is applicable. Observe the different levels of control we allow along each direction: the more ``hyperbolic" a direction is, the less restrictive are the conditions on the perturbations along such a direction.  
\end{example}

\subsection{Proof of Theorem \ref{theo: lienarization}}
In this subsection we present the proof of Theorem \ref{theo: lienarization}.  For the sake of clarity, we will divide it into several steps. 

Let $\mathcal Y$ denote the space of all sequences $\mathbf h=(h_n)_{n\in \N}$ of continuous maps $h_n:X\to X$ such that
\[
\lVert \mathbf h\rVert_{\mathcal Y} :=\sup_{n\in \N} \lVert h_n\rVert_\infty <+\infty. 
\]
It is easy to verify that $(\mathcal Y, \lVert \cdot \rVert_{\mathcal Y})$ is a Banach space. 

\subsection*{Construction of maps $H_n$:} 
Let us consider the operator $\mathcal T\colon \mathcal Y \to \mathcal Y$ given by
\begin{equation*}
(\mathcal T\mathbf h)_n (x) = \sum_{k\in K^s\cup K^u}(\mathcal T_k\mathbf h)_n (x),
\end{equation*}
where 
\begin{itemize}
\item for $k\in K^s$, we set $(\mathcal T_k\mathbf h)_0 (x)=0 $ and 
\begin{equation*}
(\mathcal T_k\mathbf h)_n (x) = \sum_{l=1}^n \cA(n, l)P^k_l (f_{l-1}(\cA(l-1, n)x+h_{l-1}(\cA(l-1, n)x))),
\end{equation*}
for $n\geq 1$;
\item for $k\in K^u$, we define
\begin{equation*}
(\mathcal T_k\mathbf h)_n (x) =-\sum_{l=n+1}^\infty \cA(n, l)P^k_l(f_{l-1}(\cA(l-1, n)x+h_{l-1}(\cA(l-1, n)x))),
\end{equation*}
for every $n\in \N$ and $x\in X$.
\end{itemize}

Since $P^k_l=P_l^k P^k_l$, we have that 
\begin{equation*}
\begin{split}
|(\mathcal T\mathbf h)_n (x)| &\leq \sum_{k\in K^s\cup K^u}|(\mathcal T_k\mathbf h)_n (x) |\\
&\leq \sum_{k\in K^s}|(\mathcal T_k\mathbf h)_n (x)|+ \sum_{k\in K^u}|(\mathcal T_k\mathbf h)_n (x)|\\
&\leq \sum_{k\in K^s} \sum_{l=1}^n \|\cA(n, l)P^k_l\|\cdot  \|P^k_lf_{l-1}\|_\infty\\
&\phantom{=}+\sum_{k\in K^u}\sum_{l=n+1}^\infty \|\cA(n, l)P^k_l\| \cdot \|P^k_lf_{l-1}\|_\infty,\\
\end{split}
\end{equation*}
which combined with \eqref{eq: multi c1}, \eqref{eq: multi c3} and \eqref{bound of f} implies that
\begin{equation*}
\begin{split}
\sup_{n\in \N}\|(\mathcal T\mathbf h)_n \|_\infty &<+\infty.\\
\end{split}
\end{equation*}
Moreover, one can easily see that for every $\mathbf{h}\in \mathcal Y$ and $n\in \N$, $(\mathcal{T}\mathbf{h})_n$ is continuous. 
Hence, $\mathcal T\colon \mathcal Y \to \mathcal Y$ is well-defined. We  now claim that $\mathcal T\colon \mathcal Y \to \mathcal Y$ is a contraction. Indeed, take $\mathbf h^i=(h_n^i)_{n\in \Z} \in \mathcal Y$, $i=1, 2$. Using~\eqref{eq: multi c2} and \eqref{Lipschitz} we get that for each $k\in K^s$,
\[
\begin{split}
| (\mathcal T_k\mathbf h^1)_n (x)-(\mathcal T_k\mathbf h^2)_n (x)| &\leq \sum_{l=1}^n \|\cA(n,l)P^k_l\|\mu^k_l \lVert h_{l-1}^1-h_{l-1}^2\rVert_\infty \\
&\leq \lambda_k \lVert \mathbf h^1-\mathbf h^2\rVert_{\mathcal{Y}}, \\
\end{split}
\]
for $x\in X$ and $n\in \N$. Similarly, for each $k\in K^u$,
\[
\begin{split}
| (\mathcal T_k\mathbf h^1)_n (x)-(\mathcal T_k\mathbf h^2)_n (x)| &\leq \sum_{l=n+1}^\infty \|\cA(n,l)P^k_l\|\mu^k_l \lVert h_{l-1}^1-h_{l-1}^2\rVert_\infty \\
&\leq \lambda_k \lVert \mathbf h^1-\mathbf h^2\rVert_{\mathcal{Y}}, \\
\end{split}
\]
for $x\in X$ and $n\in \N$.
Consequently,
\begin{displaymath}
\begin{split}
| (\mathcal T\mathbf h^1)_n (x)-(\mathcal T\mathbf h^2)_n (x)| &\leq \sum_{k\in K^s\cup K^u}\lambda_k \lVert \mathbf h^1-\mathbf h^2\rVert_{\mathcal{Y}}  \\
\end{split}
\end{displaymath}
for every $x\in X$ and $n\in \N$ and thus,
\begin{displaymath}
\begin{split}
\lVert \mathcal T\mathbf h^1-\mathcal T\mathbf h^2 \rVert_{\mathcal{Y}} &\leq \sum_{k\in K^s\cup K^u}\lambda_k \lVert \mathbf h^1-\mathbf h^2\rVert_{\mathcal{Y}}.  \\
\end{split}
\end{displaymath}
Hence, by~\eqref{contr} we conclude that $\mathcal T$ is a contraction. Therefore, $\mathcal T$ has a unique fixed point $\mathbf h=(h_n)_{n\in \N}\in \mathcal Y$.  Thus, we have that
\begin{equation}\label{eq: aux conjd}
h_{n+1}(A_n x) =(\mathcal T \mathbf h)_{n+1}(A_n x) =\sum_{k\in K^s\cup K^u}(\mathcal T_k\mathbf h)_{n+1} (A_n x),
\end{equation}
for $x\in X$ and $n\in \N$.
Now, for $k\in K^s$, $x\in X$ and $n\geq 1$, we have that 
\begin{displaymath}
\begin{split}
&(\mathcal T_k\mathbf h)_{n+1} (A_n x)\\
&= \sum_{l=1}^{n+1} \cA(n+1, l)P^k_l (f_{l-1}(\cA(l-1, n+1)A_nx+h_{l-1}(\cA(l-1, n+1)A_nx)))\\
&=\sum_{l=1}^{n+1} \cA(n+1, l)P^k_l (f_{l-1}(\cA(l-1, n)x+h_{l-1}(\cA(l-1, n)x)))\\
&=A_n\sum_{l=1}^{n} \cA(n, l)P^k_l (f_{l-1}(\cA(l-1, n)x+h_{l-1}(\cA(l-1, n)x)))+ P^k_{n+1} (f_{n}(x+h_{n}(x))) \\
&=A_n(\mathcal T_k\mathbf h)_{n} ( x)+ P^k_{n+1} (f_{n}(x+h_{n}(x))).
\end{split}
\end{displaymath}
Similarly, for $k\in K^s$, $x\in X$ and $n=0$, we observe that
\begin{displaymath}
(\mathcal T_k\mathbf h)_{1} (A_0 x)=P^k_{1} (f_{0}(x+h_{0}(x)))=A_n(\mathcal T_k\mathbf h)_{0} ( x)+P^k_{1} (f_{0}(x+h_{0}(x))).
\end{displaymath}
Finally, for $k\in K^u$, $x\in X$ and $n\in \N$, we have that 
\begin{displaymath}
\begin{split}
&(\mathcal T_k\mathbf h)_{n+1} (A_n x)\\
&= -\sum_{l=n+2}^{\infty} \cA(n+1, l)P^k_l (f_{l-1}(\cA(l-1, n+1)A_nx+h_{l-1}(\cA(l-1, n+1)A_nx)))\\
&=-\sum_{l=n+2}^{\infty} \cA(n+1, l)P^k_l (f_{l-1}(\cA(l-1, n)x+h_{l-1}(\cA(l-1, n)x)))\\
&=-A_n\sum_{l=n+1}^{\infty} \cA(n, l)P^k_l (f_{l-1}(\cA(l-1, n)x+h_{l-1}(\cA(l-1, n)x)))+ P^k_{n+1} (f_{n}(x+h_{n}(x)))\\
&=A_n(\mathcal T_k\mathbf h)_{n} ( x)+ P^k_{n+1} (f_{n}(x+h_{n}(x))).
\end{split}
\end{displaymath}

Combining these observations with \eqref{split}, \eqref{eq: aux conjd} and the fact that $\mathbf h$ is a fixed point of $\mathcal T$, we obtain that
\[
h_{n+1}(A_n x) =A_n h_n(x)+f_n(x+h_n(x))-\sum_{k\in K^c}P^k_{n+1} (f_{n}(x+h_{n}(x))),
\]
for $n\in \N$ and $x\in X$. Consequently, defining $H_n=\Id+h_n$, $n\in \N$, and $\tau_n:X\to \bigoplus_{k\in K^c}X_k(n+1)$ by $\tau_n(x)=-\sum_{k\in K^c}P^k_{n+1} (f_{n}(x))$ for $n\in \N$, we get that~\eqref{lin1} holds.

\subsection*{Construction of maps  $\bar{H}_n$:}
We now consider $\mathbf{\bar h}=(\bar h_n)_{n\in \N} \in \mathcal Y$ given by 
\begin{equation*}
\bar h_n (x) = \sum_{k\in K^s\cup K^u}\bar h^k_n (x) ,
\end{equation*}
where
\begin{itemize}
\item for $k\in K^s$, we set $\bar h^k_0 (x)=0$ and 
\begin{displaymath}
\bar h^k_n (x)= -\sum_{l=1}^{n} \cA(n, l)P^k_lf_{l-1}(\mathcal F(l-1, n)x) \quad \text{for $n\ge 1$,}
\end{displaymath}
where
\[
\mathcal F(m, n)=\begin{cases}
F_{m-1}\circ \ldots F_n & \text{for $m>n$;}\\
\Id & \text{for $m=n$;}\\
F_{m+1}^{-1}\circ \ldots \circ F_n^{-1} & \text{for $m<n$,}
\end{cases}
\]
and $F_n=A_n+f_n$, $n\in \N$;
\item  for $k\in K^u$ and $n\in \N$,
\begin{displaymath}
\bar h^k_n (x)= \sum_{l=n+1}^{\infty} \cA(n, l)P^k_lf_{l-1}(\mathcal F(l-1, n)x).
\end{displaymath}

\end{itemize}
It follows easily from \eqref{eq: multi c1}, \eqref{eq: multi c3} and \eqref{bound of f} that indeed $ \mathbf{\bar h}\in \mathcal Y$. Moreover, we observe that given $x\in X$ and $k\in K^s$, we have that
\[
\begin{split}
&\bar{h}^k_{n+1}(F_n(x))  \\
&=-\sum_{l=1}^{n+1} \cA(n+1, l)P^k_lf_{l-1}(\mathcal F(l-1, n+1)F_n(x))\\
&=-\sum_{l=1}^{n+1} \cA(n+1, l)P^k_lf_{l-1}(\mathcal F(l-1, n)x)\\
&=-A_n\sum_{l=1}^{n} \cA(n, l)P^k_lf_{l-1}(\mathcal F(l-1, n)x)-P^k_{n+1}f_n(x)\\
&=A_n\bar{h}^k_n(x)-P^k_{n+1}f_n(x),\\
\end{split}
\]
for $n\geq 1$. Similarly, for $x\in X$ and $n=0$,  we observe that
\[
\begin{split}
\bar{h}^k_{1}(F_0(x))=-P^k_1f_{0}(x)=A_0\bar{h}^k_0(x)-P^k_{1}f_0(x).
\end{split}
\]
Moreover, for $k\in K^u$, $x\in X$ and $n\in \N$, we have that 
\[
\begin{split}
&\bar{h}^k_{n+1}(F_n(x))  \\
&=\sum_{l=n+2}^{\infty} \cA(n+1, l)P^k_lf_{l-1}(\mathcal F(l-1, n+1)F_n(x))\\
&=\sum_{l=n+2}^{\infty} \cA(n+1, l)P^k_lf_{l-1}(\mathcal F(l-1, n)x)\\
&=A_n\sum_{l=n+1}^{\infty} \cA(n, l)P^k_lf_{l-1}(\mathcal F(l-1, n)x)-P^k_{n+1}f_n(x)\\
&=A_n\bar{h}^k_n(x)-P^k_{n+1}f_n(x).\\
\end{split}
\]
Consequently, using \eqref{split} and recalling the definition of $\bar{\mathbf{h}}$, it follows that for every $n\in \N$ and $x\in X$,
\begin{displaymath}
\bar h_{n+1}(F_n(x))=A_n\bar{h}_n(x)-f_n(x)+\sum_{k\in K^c}P^k_{n+1}f_{n}(x).
\end{displaymath}
Thus, defining $\bar{H}_n=\Id+\bar{h} _n$, $n\in \N$, and $\bar \tau_n:X\to \bigoplus_{k\in K^c}X_k(n+1)$ by $\bar \tau_n(x)=\sum_{k\in K^c}P^k_{n+1} (f_{n}(x))$ for $n\in \N$, we conclude that~\eqref{lin2} holds.

\subsection*{The cases when $K^c=\emptyset$ and $P_n^kf_{n-1}\equiv 0$ for every $k\in K^c$:}
Suppose that either $K^c=\emptyset$ or $P_n^kf_{n-1}\equiv 0$ for every $k\in K^c$ and $n\in \N$. Hence, we have that  $\tau_n=\overline{\tau}_n=0$ for every $n\in \N$. In particular, \eqref{lin1} and \eqref{lin2} imply that
\begin{equation}\label{eq: conjug proof}
H_{n+1}\circ A_n=(A_n+f_n)\circ H_n \text{ and } \bar{H}_{n+1}\circ (A_n+f_n)=A_n \circ \bar{H}_n,
\end{equation}
for every $n\in \N$. Hence, \eqref{eq: linearization theo1} holds. Moreover,
 it follows easily from~\eqref{eq: conjug proof} that
\begin{equation}\label{eq: Hn sol of Fn}
H_n(\cA(n,m)x)=\mathcal F(n,m)H_m(x)
\end{equation}
and
\begin{equation}\label{eq: barHn sol of An}
\bar H_n(\mathcal F(n,m)x)=\cA(n,m)\bar H_m(x),
\end{equation}
for every $m,n\in\N$.

Recalling the definitions of $\bar H_n$ and $H_n$ we get that for every $n\geq 1$ and $x\in X$,
\begin{equation}\label{2153}
\begin{split}
\bar H_n(H_{n}(x))&=H_{n}(x)+\bar h_n(H_{n}(x))\\
&=x +h_{n}(x)+\bar h_n(H_{n}(x))\\
&=x +\sum_{k\in K^s}\sum_{l=1}^n \cA(n, l)P^k_l (f_{l-1}(\cA(l-1, n)x+h_{l-1}(\cA(l-1, n)x)))\\
&\phantom{=} -\sum_{k\in K^u}\sum_{l=n+1}^\infty \cA(n, l)P^k_l(f_{l-1}(\cA(l-1, n)x+h_{l-1}(\cA(l-1, n)x)))\\
&\phantom{=} -\sum_{k\in K^s}\sum_{l=1}^{n} \cA(n, l)P^k_lf_{l-1}(\mathcal F(l-1, n)H_{n}(x))\\
&\phantom{=}+ \sum_{k\in K^u}\sum_{l=n+1}^{\infty} \cA(n, l)P^k_lf_{l-1}(\mathcal F(l-1, n)H_{n}(x)).\\
\end{split}
\end{equation}
Now, by \eqref{eq: Hn sol of Fn} it follows that
\begin{equation*}
\begin{split}
\mathcal F(l-1, n)H_{n}(x)&=H_{l-1}(\cA(l-1,n)x)\\
&=\cA(l-1,n)x+h_{l-1}(\cA(l-1,n)x),
\end{split}
\end{equation*}
which combined with~\eqref{2153} implies that $\bar H_n(H_{n}(x))=x$ for every $x\in X$. The case when $n=0$ can be treated similarly.

Our objective now is to show that $ H_n(\bar H_{n}(x))=x$ for every $x\in X$ and $n\in \N$. We start by observing that
\begin{equation*}
\begin{split}
 H_n(\bar H_{n}(x))&=\bar H_{n}(x)+ h_n(\bar H_{n}(x))\\
&=x +\bar h_{n}(x)+ h_n(\bar H_{n}(x)).\\
\end{split}
\end{equation*}
Consequently,
\begin{equation}\label{eq: aux Hn barHn}
\begin{split}
 H_n(\bar H_{n}(x))-x=\bar h_{n}(x)+ h_n(\bar H_{n}(x)).
\end{split}
\end{equation}
By analyzing the right-hand side of~\eqref{eq: aux Hn barHn},  we have that 
\begin{equation*}
\begin{split}
&\bar h_{n}(x)+ h_n(\bar H_{n}(x))\\
&=-\sum_{k\in K^s}\sum_{l=1}^{n} \cA(n, l)P^k_lf_{l-1}(\mathcal F(l-1, n)x)\\
&\phantom{=}+ \sum_{k\in K^u}\sum_{l=n+1}^{\infty} \cA(n, l)P^k_lf_{l-1}(\mathcal F(l-1, n)x)\\
&\phantom{=}+\sum_{k\in K^s}\sum_{l=1}^n \cA(n, l)P^k_l (f_{l-1}(\cA(l-1, n)\bar H_{n}(x)+h_{l-1}(\cA(l-1, n)\bar H_{n}(x))))\\
&\phantom{=} -\sum_{k\in K^u}\sum_{l=n+1}^\infty \cA(n, l)P^k_l(f_{l-1}(\cA(l-1, n)\bar H_{n}(x)+h_{l-1}(\cA(l-1, n)\bar H_{n}(x)))),\\
\end{split}
\end{equation*}
for $x\in X$ and $n\ge 1$.
On the other hand, by using \eqref{eq: barHn sol of An} we have that
\begin{equation*}
\begin{split}
\cA(l-1, n)\bar H_{n}(x)+h_{l-1}(\cA(l-1, n)\bar H_{n}(x))&=H_{l-1}(\cA(l-1, n)\bar H_{n}(x))\\
&=H_{l-1}(\bar H_{l-1}(\mathcal F(l-1,n)x)).
\end{split}
\end{equation*}
Thus, combining the previous observations we get that
\begin{equation*}
\begin{split}
&|\bar h_{n}(x)+ h_n(\bar H_{n}(x))|\\
&\leq\sum_{k\in K^s}\sum_{l=1}^{n}  \|\cA(n, l)P^k_l\| \cdot |P^k_lf_{l-1}(H_{l-1}(\bar H_{l-1}(\mathcal F(l-1,n)x)))-P^k_lf_{l-1}(\mathcal F(l-1, n)x)|\\
&\phantom{=}+ \sum_{k\in K^u}\sum_{l=n+1}^{\infty}  \|\cA(n, l)P^k_l\| \cdot  |P^k_lf_{l-1}(\mathcal F(l-1, n)x)-P^k_lf_{l-1}(H_{l-1}(\bar H_{l-1}(\mathcal F(l-1,n)x)))|\\
&\leq \sum_{k\in K^s}\sum_{l=1}^{n} \|\cA(n, l)P^k_l\|\mu^k_{l-1}|H_{l-1}(\bar H_{l-1}(\mathcal F(l-1,n)x)))-\mathcal F(l-1, n)x|\\
&\phantom{=}+ \sum_{k\in K^u}\sum_{l=n+1}^{\infty} \|\cA(n, l)P^k_l\|\mu^k_{l-1}|H_{l-1}(\bar H_{l-1}(\mathcal F(l-1,n)x)))-\mathcal F(l-1, n)x|.\\
\end{split}
\end{equation*}
Therefore, using \eqref{eq: aux Hn barHn} it follows that
\begin{equation}\label{2203}
\begin{split}
& |H_n(\bar H_{n}(x))-x|\\
&\leq \sum_{k\in K^s}\sum_{l=1}^{n} \|\cA(n, l)P^k_l\|\mu^k_{l-1}|H_{l-1}(\bar H_{l-1}(\mathcal F(l-1,n)x)))-\mathcal F(l-1, n)x|\\
&\phantom{=}+ \sum_{k\in K^u}\sum_{l=n+1}^{\infty}  \|\cA(n, l)P^k_l|\mu^k_{l-1}|H_{l-1}(\bar H_{l-1}(\mathcal F(l-1,n)x)))-\mathcal F(l-1, n)x|.\\
\end{split}
\end{equation}
Now, since $\textbf{h}=(h_n)_{n\in\N}\in \mathcal Y$ and $\bar{\textbf{h}}=(\bar{h}_n)_{n\in\N}\in \mathcal Y$, it follows by \eqref{eq: aux Hn barHn} that $\textbf{H}\circ \bar{\textbf{H}}-\Id:=(H_n\circ \bar{H}_n-\Id)_{n\in\N}\in \mathcal Y$,
which combined with \eqref{eq: multi c2}, \eqref{eq: multi c4} and~\eqref{2203} implies that
\begin{displaymath}
\|\textbf{H}\circ \bar{\textbf{H}}-\Id\|_{\mathcal{Y}}\leq \sum_{k\in K^s\cup K^u}\lambda_k\|\textbf{H}\circ \bar{\textbf{H}}-\Id\|_{\mathcal{Y}}.
\end{displaymath}
Thus, from~\eqref{contr} it follows that $\|\textbf{H}\circ \bar{\textbf{H}}-\Id\|_{\mathcal{Y}}=0$, and consequently $H_n(\bar H_{n}(x))=x$ for every $x\in X$ and $n\in \N$.

Summarizing, we have proved that~\eqref{eq: H inverse theo1} holds. In particular, we conclude that 
 $H_n$ and $\bar H_n$ are homeomorphisms for each $n\in \N$. The proof of Theorem \ref{theo: lienarization} is completed.

\section{The case of continuous time}\label{CT}
The purpose of this section is to establish the version of Theorem~\ref{theo: lienarization} for the case of continuous time. Let $A\colon [0, \infty)\to \mathcal B(X)$ and $f\colon [0, \infty)\times X\to X$ be continuous maps. We consider the associated semilinear differential equation
\begin{equation}\label{sde}
x'=A(t)x+f(t,x) \quad t\ge 0,
\end{equation}
as well as the associated linear equation
\begin{equation}\label{lde}
x'=A(t)x \quad t\ge 0.
\end{equation}
By $T(t,s)$ we will denote the evolution family associated to~\eqref{lde}. Furthermore, $U(t,s)$ will denote the nonlinear evolution family corresponding to~\eqref{sde}, i.e. $U(t,s)v=x(t)$, where $x\colon [0, \infty)\to X$ is the solution of~\eqref{sde} such that 
$x(s)=v$.

Let $K$ be as in Subsection~\ref{Multi}. Suppose that for each $t\ge 0$ and $k\in K$ there is a projection $P^k(t)$ on $X$ such that:
\begin{itemize}
\item $\sum_{k\in K} P^k(t)=\Id$;
\item $P^k(t)P^l(t)=0$ for $k, l\in K$, $k\neq l$;
\item for $k\in K$, $t\mapsto P^k(t)$ is measurable.
\end{itemize}
Furthermore, we assume that there are Borel measurable functions $\mu^k, \nu^k \colon [0, \infty)\to [0, \infty)$ and positive numbers $\lambda_k>0$, $k\in K^s \cup K^u$ such that:
\begin{itemize}
\item for $k\in K^s$, 
\begin{equation}\label{c1}
\sup_t \int_0^t\|T(t,s)P^k(s)\|\nu^k(s)\, ds<+\infty,
\end{equation}
and
\begin{equation}\label{c2}
\sup_t \int_0^t\|T(t,s)P^k(s)\|\mu^k(s)\, ds\le \lambda_k;
\end{equation}
\item for $k\in K^u$, 
\begin{equation}\label{c3}
\sup_t \int_t^\infty \|T(t,s)P^k(s)\|\nu^k(s)\, ds<+\infty,
\end{equation}
and
\begin{equation}\label{c4}
\sup_t \int_t^\infty \|T(t,s)P^k(s)\|\mu^k(s)\, ds\le \lambda_k.
\end{equation}
\end{itemize}
The following is the version of Theorem~\ref{theo: lienarization} in the present setting.
\begin{theorem}\label{THM2}
Assume that the following conditions hold:
\begin{itemize}
\item for $k\in K^s\cup K^u$ and $t\ge 0$,
\begin{equation}\label{c1c}
\|P^k(t)f(t, \cdot)\|_\infty \le \nu^k(t);
\end{equation}
\item for $k\in K^s\cup K^u$, $t\ge 0$ and $x, y\in X$,
\begin{equation}\label{c2c}
|P^k(t)f(t, x)-P^k(t)f(t, y)| \le \mu^k(t)|x-y|;
\end{equation}
\item \eqref{contr} holds.
\end{itemize}
Then, 
\begin{itemize}
\item there exists a continuous map $H\colon [0, \infty)\times X\to X$ such that if $t\mapsto x(t)$ is a solution of~\eqref{lde}, then $t\mapsto H(t, x(t))$ is a solution of
\begin{equation}\label{sde1}
x'=A(t)x+\sum_{k\in K^s\cup K^u} P^k(t)f(t,x);
\end{equation}
\item there exists a continuous map $\bar H\colon [0, \infty)\times X\to X$ such that if $t\mapsto y(t)$ is a solution of~\eqref{sde}, then $t\mapsto \bar H(t, y(t))$ is a solution of
\begin{equation}\label{sde2}
x'=A(t)x+\sum_{k\in K^c} P^k(t)f(t,y(t)));
\end{equation}
\item we have that
\begin{equation}\label{Bound}
\sup_{t} \|H(t, \cdot)-\Id \|_\infty<+\infty \quad \text{and} \quad \sup_{t} \|\bar H(t, \cdot)-\Id\|_\infty<+\infty.
\end{equation}
\end{itemize}
Moreover, in the case when $K^c=\emptyset$ or $P^k(t)f(t, \cdot)\equiv 0$ for $t\ge 0$ and $k\in K^c$, then $H(t, \cdot)$ and $\bar H(t, \cdot)$ are homeomorphisms for each $t\ge 0$ satisfying
\begin{equation}\label{HH}
H(t, \bar H(t, x))=\bar H(t, H(t, x))=x,
\end{equation}
\begin{equation}\label{213}
H(t, T(t,s)x)=U(t,s)H(s, x) \quad \text{and} \quad \bar H(t, U(t,s)x)=T(t,s)\bar H(s,x),
\end{equation}
for $t, s\ge 0$ and $x\in X$.
\end{theorem}

\begin{proof}
We follow closely the proof of Theorem~\ref{theo: lienarization}.
Let $\mathcal Y$ denote the space of all continuous functions  $h\colon [0, \infty)\times X\to X$ such that 
\[
\|h\|_{\mathcal Y}:=\sup_{t\ge 0} \|h(t, \cdot)\|_\infty=\sup_{t, x}|h(t,x)| <+\infty.
\]
Then, $(\mathcal Y, \| \cdot \|_{\mathcal Y})$ is a Banach space. We define an operator $\mathcal T\colon \mathcal Y\to \mathcal Y$ by 
\[
(\mathcal Th) (t, x) = \sum_{k\in K^s\cup K^u}(\mathcal T_k h) (t, x),
\]
where 
\begin{itemize}
\item for $k\in K^s$, we set
\begin{equation*}
(\mathcal T_kh) (t, x) = \int_0^t T(t, s)P^k(s) (f(s, T(s, t)x+h(s, T(s, t)x)))\, ds,
\end{equation*}
for every $t\ge 0$ and $x\in X$;
\item for $k\in K^u$, we set 
\begin{equation*}
(\mathcal T_k h)(t, x) =-\int_{t}^\infty T(t, s)P^k(s)(f(s, T(s, t)x+h(s, T(s, t)x)))\, ds,
\end{equation*}
for every $t\ge 0$ and $x\in X$.
\end{itemize}

Observe that 
\begin{equation*}
\begin{split}
|(\mathcal T h) (t, x)| &\leq \sum_{k\in K^s\cup K^u}|(\mathcal T_k h) (t, x) |\\
&\leq \sum_{k\in K^s}|(\mathcal T_kh) (t, x)|+ \sum_{k\in K^u}|(\mathcal T_k h) (t, x)|\\
&\leq \sum_{k\in K^s} \int_{0}^t \|T(t, s)P^k(s)\|\cdot  \|P^k(s)f(s, \cdot)\|_\infty\, ds\\
&\phantom{=}+\sum_{k\in K^u}\int_{t}^\infty \|T(t, s)P^k(s)\| \cdot \|P^k(s)f(s, \cdot)\|_\infty\, ds,\\
\end{split}
\end{equation*}
which combined with \eqref{c1}, \eqref{c3} and \eqref{c1c} implies that
\begin{equation*}
\begin{split}
\sup_{t\ge 0}\|(\mathcal T h)(t, \cdot) \|_\infty &<+\infty.\\
\end{split}
\end{equation*}
This easily implies that $\mathcal Th\in \mathcal Y$.  Take now $h^1, h^2\in \mathcal Y$. By using~\eqref{c2} and \eqref{c2c} we get that for each $k\in K^s$,
\[
\begin{split}
| (\mathcal T_k h^1) (t, x)-(\mathcal T_k h^2) (t, x)| &\leq \int_{0}^t \|T(t,s)P^k(s) \|\mu^k(s) \lVert h^1(s, \cdot)-h^2(s, \cdot)\rVert_\infty \, ds\\
&\leq \lambda_k \lVert h^1-h^2\rVert_{\mathcal{Y}}, \\
\end{split}
\]
for $x\in X$ and $t\ge 0$. Similarly, by using~\eqref{c4} and~\eqref{c2c}, we have that  for each $k\in K^u$,
\[
\begin{split}
| (\mathcal T_kh^1) (t, x)-(\mathcal T_k h^2) (t, x)| &\leq \int_{t}^\infty \|T(t,s)P^k(s)\|\mu^k(s) \lVert h^1(s, \cdot)-h^2(s, \cdot)\rVert_\infty\, ds \\
&\leq \lambda_k \lVert h^1- h^2\rVert_{\mathcal{Y}}, \\
\end{split}
\]
for $x\in X$ and $t\ge 0$. Consequently,
\begin{displaymath}
\begin{split}
| (\mathcal Th^1) (t, x)-(\mathcal Th^2) (t, x)| &\leq \sum_{k\in K^s\cup K^u}\lambda_k \lVert h^1- h^2\rVert_{\mathcal{Y}}  \\
\end{split}
\end{displaymath}
for every $x\in X$ and $t\ge 0$ and thus,
\begin{displaymath}
\begin{split}
\lVert \mathcal Th^1-\mathcal Th^2 \rVert_{\mathcal{Y}} &\leq \sum_{k\in K^s\cup K^u}\lambda_k \lVert  h^1-h^2\rVert_{\mathcal{Y}}.  \\
\end{split}
\end{displaymath}
Hence, by~\eqref{contr} we conclude that $\mathcal T$ is a contraction. Therefore, $\mathcal T$ has a unique fixed point $h\in \mathcal Y$.  Therefore,
\begin{equation}\label{eq: aux conj}
h(t, T(t,s) x) =(\mathcal T h)(t, T(t,s)x) =\sum_{k\in K^s\cup K^u}(\mathcal T_kh) (t, T(t,s) x),
\end{equation}
for $x\in X$ and $t, s\ge 0$.
Now, for $k\in K^s$, $x\in X$ and $t, s\ge 0$, we have that 
\begin{displaymath}
\begin{split}
&(\mathcal T_kh) (t, T(t,s) x)\\
&= \int_{0}^{t} T(t, r)P^k(r)(f(r, T(r,t)T(t,s)x+h(r,T(r, t)T(t,s)x)))\, dr\\
&=\int_{0}^{t} T(t, r)P^k(r)(f(r, T(r,s)x+h(r, T(r, s)x)))\, dr\\
&=T(t,s) \int_{0}^{s} T(s, r)P^k(r)(f(r, T(r,s)x+h(r, T(r, s)x)))\, dr\\
&\phantom{=}+ \int_s^t T(t, r)P^k(r)(f(r, T(r,s)x+h(r, T(r, s)x)))\, dr \\
&=T(t,s)(\mathcal T_k h) (s, x)+ \int_s^t T(t, r)P^k(r)(f(r, T(r,s)x+h(r, T(r, s)x)))\, dr.
\end{split}
\end{displaymath}
Similarly, 
 for $k\in K^u$, $x\in X$ and $t, s\ge 0$, we have that 
\begin{displaymath}
\begin{split}
&(\mathcal T_k h) (t, T(t,s) x)\\
&= -\int_{t}^{\infty} T(t,r)P^k(r) (f(r,T(r, t)T(t,s)x+h(r,T(r, t)T(t,s)x)))\, dr\\
&=-\int_{t}^{\infty} T(t, r)P^k(r) (f(r, T(r, s)x+h(r, T(r, s)x)))\, dr\\
&=-T(t,s)\int_{s}^{\infty} T(s, r)P^k(r) (f(r,T(r,s)x+h(r, T(r, s)x)))\, dr\\
&\phantom{=}+ \int_s^t T(t, r)P^k(r) (f(r, T(r, s)x+h(r, T(r, s)x)))\, dr\\
&=T(t,s)(\mathcal T_k h) (s, x)+ \int_s^t T(t, r)P^k(r) (f(r, T(r, s)x+h(r, T(r, s)x)))\, dr.
\end{split}
\end{displaymath}
Combining these observations with  \eqref{eq: aux conj} and the fact that $h$ is a fixed point of $\mathcal T$, we obtain that
\[
\begin{split}
h(t, T(t,s) x) &=T(t,s) h(s, x)+\int_s^t T(t, r) (f(r, T(r, s)x+h(r, T(r, s)x)))\, dr \\
&\phantom{=}-\sum_{k\in K^c} \int_s^t T(t, r)P^k(r) (f(r, T(r, s)x+h(r, T(r, s)x)))\, dr,
\end{split}
\]
for $t, s\ge 0$ and $x\in X$.  By differentiating the above equality, we easily conclude that if $t\mapsto x(t)$ is a solution of~\eqref{lde}, then $t\mapsto H(t, x(t))$ is a solution of~\eqref{sde1}, 
where $H(t,x):=x+h(t,x)$.

We now consider $\bar h \in \mathcal Y$ given by 
\begin{equation*}
\bar h (t, x) = \sum_{k\in K^s\cup K^u}\bar h^k (t, x) ,
\end{equation*}
where
\begin{itemize}
\item for $k\in K^s$ and  $t\geq 0$,
\begin{displaymath}
\bar h^k (t, x):= -\int_{0}^{t} T(t, s)P^k(s) f(s, U(s, t)x)\, ds;
\end{displaymath}
\item  for $k\in K^u$ and $t\ge 0$,
\begin{displaymath}
\bar h^k (t, x):= \int_{t}^{\infty} T(t, s)P^k(s) f(s, U(s,t)x)\, ds.
\end{displaymath}

\end{itemize}
It follows easily from \eqref{c1}, \eqref{c3} and \eqref{c1c} that indeed $\bar h\in \mathcal Y$. Moreover, we observe that given $x\in X$ and $k\in K^s$, we have that
\[
\begin{split}
&\bar{h}^k(t, U(t,s)x)  \\
&=-\int_{0}^{t} T(t,r)P^k(r) f(r, U(r, t)U(t,s)x)\, dr\\
&=-\int_{0}^{t} T(t,r)P^k(r) f(r, U(r,s)x)\, dr\\
&=-T(t,s)\int_{0}^{s} T(s,r)P^k(r) f(r, U(r,s)x)\, dr-\int_{s}^{t} T(t,r)P^k(r) f(r, U(r,s)x)\, dr\\
&=T(t,s) \bar{h}^k(s, x)-\int_{s}^{t} T(t,r)P^k(r) f(r, U(r,s)x)\, dr,\\
\end{split}
\]
for $t, s\ge 0$ and $x\in X$. 
Moreover, for $k\in K^u$, $x\in X$ and $t, s \ge 0$, we have that 
\[
\begin{split}
&\bar{h}^k(t, U(t,s)x)  \\
&=\int_{t}^{\infty} T(t,r)P^k(r)f(r, U(r,t)U(t,s)x)\, dr\\
&=\int_{t}^{\infty} T(t,r)P^k(r)f(r, U(r,s)x)\, dr\\
&=T(t,s) \int_{s}^{\infty} T(s,r)P^k(r)f(r, U(r,s)x)\, dr-\int_s^t T(t,r)P^k(r)f(r, U(r,s)x)\, dr\\
&=T(t,s)\bar{h}^k(s, x)-\int_s^t T(t,r)P^k(r)f(r, U(r,s)x)\, dr.\\
\end{split}
\]
Consequently,  it follows that for $t, s\ge 0$ we have that 
\begin{displaymath}
\begin{split}
\bar h(t, U(t,s)x)&=T(t,s)\bar{h}(s, x)-\int_s^t T(t,r)f(r, U(r,s)x)\, dr \\
&\phantom{=}+\sum_{k\in K^c}\int_s^t T(t,r)P^k(r)f(r, U(r,s)x)\, dr.
\end{split}
\end{displaymath}
From this we easily conclude that if $t\mapsto y(t)$ is a solution of~\eqref{sde}, then  $t\mapsto \bar H(t, y(t))$ is a solution of~\eqref{sde2}, 
where $\bar H(t,x)=x+\bar h(t,x)$. Finally, we observe that  since $h, \bar h\in \mathcal Y$, we have that~\eqref{Bound} holds.

Suppose now that either $K^c=\emptyset$ or $P^k(t)f(t, \cdot)\equiv 0$ for every $k\in K^c$ and $t\in \R$. From the previous observations, we conclude that~\eqref{213} holds 
for $t, s\ge 0$ and $x\in X$. Moreover, for every $t\ge 0$ and $x\in X$, we have that 
\[
\begin{split}
\bar H(t, H(t, x))&=H(t, x)+\bar h(t, H(t, x))\\
&=x +h(t, x)+\bar h(t, H(t, x))\\
&=x +\sum_{k\in K^s}\int_{0}^t T(t,s)P^k(s) (f(s, T(s, t)x+h(s, T(s, t)x)))\, ds\\
&\phantom{=} -\sum_{k\in K^u}\int_{t}^\infty T(t, s)P^k(s)(f(s, T(s, t)x+h(s, T(s, t)x)))\, ds\\
&\phantom{=} -\sum_{k\in K^s}\int_{0}^{t} T(t, s)P^k(s)f(s, U(s, t)H(t, x))\, ds\\
&\phantom{=}+ \sum_{k\in K^u}\int_{t}^{\infty} T(t, s)P^k(s)f(s, U(s, t)H(t, x))\, ds.\\
\end{split}
\]
By applying~\eqref{213}, we conclude that $\bar H(t, H(t, x))=x$ for $x\in X$ and $t\ge 0$.

We now claim that  $H(t,\bar  H(t, x))=x$ for $x\in X$ and $t\ge 0$. Observe that
\begin{equation}\label{eq: aux Hn barHnc}
\begin{split}
 H(t, \bar H(t, x))-x=\bar h(t, x)+ h(t, \bar H(t, x)).
\end{split}
\end{equation}
For $t\ge 0$, we have that 
\begin{equation*}
\begin{split}
&\bar h(t, x)+ h(t, \bar H(t, x))\\
&=-\sum_{k\in K^s}\int_{0}^{t} T(t, s)P^k(s)f(s, U(s, t)x)\, ds\\
&\phantom{=}+ \sum_{k\in K^u}\int_{t}^{\infty} T(t, s)P^k(s)f(s, U(s, t)x)\, ds\\
&\phantom{=}+\sum_{k\in K^s}\int_{0}^t T(t, s)P^k(s) (f(s, T(s, t)\bar H(t, x)+h(s, T(s, t)\bar H(t,x))))\, ds\\
&\phantom{=} -\sum_{k\in K^u}\int_{t}^\infty T(t, s)P^k(s)(f(s, T(s, t)\bar H(t, x)+h(s, T(s, t)\bar H(t, x))))\, ds.\\
\end{split}
\end{equation*}
By~\eqref{213}, we have that 
\begin{equation*}
\begin{split}
T(s, t)\bar H(t, x)+h(s, T(s, t)\bar H(t, x))&=H(s, T(s, t)\bar H(t, x))\\
&=H(s, \bar H(s, U(s,t)x)),
\end{split}
\end{equation*}
and thus
\begin{equation*}
\begin{split}
&|\bar h(t, x)+ h(t, \bar H(t,x))|\\
&\leq\sum_{k\in K^s}\int_{0}^{t}  \|T(t, s)P^k(s)\| \cdot |P^k(s) f(s, H(s, \bar H(s, U(s,t)x)))-P^k(s) f(s, U(s, t)x)|\, ds\\
&\phantom{=}+ \sum_{k\in K^u}\int_{t}^{\infty}  \|T(t,s)P^k(s)\| \cdot  |P^k(s)f(s, U(s, t)x)-P^k(s) f(s, H(s, \bar H(s, U(s,t)x)))|\, ds\\
&\leq \sum_{k\in K^s}\int_{0}^{t} \|T(t, s)P^k(s)\|\mu^k(s) |H(s, \bar H(s, U(s,t)x)))-U(s, t)x|\, ds\\
&\phantom{=}+ \sum_{k\in K^u}\int_{t}^{\infty} \|T(t, s)P^k(s)\|\mu^k(s) |H(s, \bar H(s, U(s,t)x)))-U(s, t)x|\, ds.\\
\end{split}
\end{equation*}
Therefore, using \eqref{eq: aux Hn barHnc} it follows that
\begin{equation*}
\begin{split}
& |H(t, \bar H(t, x))-x|\\
&\leq \sum_{k\in K^s}\int_{0}^{t} \|T(t, s)P^k(s)\|\mu^k(s) |H(s, \bar H(s, U(s,t)x)))-U(s,t)x|\, ds\\
&\phantom{=}+ \sum_{k\in K^u}\int_{t}^{\infty}  \|T(t, s)P^k(s)|\mu^k(s) |H(s, \bar H(s, U(s,t)x)))-U(s, t)x|\, ds.\\
\end{split}
\end{equation*}
Set
\[
G(t,x)=H(t, \bar H(t,x))-x, \quad t\ge 0, \ x\in X.
\]
Now, since $h, \bar h\in \mathcal Y$, it follows from \eqref{eq: aux Hn barHnc} that $G \in \mathcal Y$,
which combined with \eqref{c2} and \eqref{c4} implies that
\begin{displaymath}
\|G\|_{\mathcal{Y}}\leq \sum_{k\in K^s\cup K^u}\lambda_k\|G\|_{\mathcal{Y}}.
\end{displaymath}
Thus, from~\eqref{contr} it follows that $G=0$, and consequently $H(t, \bar H(t, x))=x$ for every $x\in X$ and $t\ge 0$. We conclude that~\eqref{HH} holds which completes the proof of the theorem.

\end{proof}

\begin{example}
Let $X$ and $K$ be as in Example~\ref{EXM}. For $t\ge 0$ and $k\in K$, let $P^k(t)$ be the projection  onto the $k^{th}$ coordinate. Moreover, take
\[
A(t)=\text{diag}\left(-1,0,0,0,1\right) \quad t\ge 0,
\]
and consider $K^s=\{1,2\}$, $K^c=\{3\}$ and $K^u=\{4,5\}$. Take $\lambda_k=\frac{1}{5}$ for every $k\in K$ and
\begin{itemize}
\item $\nu^k(t)=1$ and $\mu^k(t)=\frac{1}{5}$, for $k\in \{1,5\}$ and  $t\ge 0$;
\item $\nu^k(t)=e^{-t}$ and $\mu^k_n=\frac{1}{5}e^{-t}$, for $k\in \{2,4\}$ and $t\ge 0$.
\end{itemize}
Let $f\colon [0, \infty)\times  X\to X$, $f=(f^1, \ldots, f^5)$  be a  continuous map such that 
\begin{itemize}
\item $\|f^k(t, \cdot)\|_\infty \le 1$ and $Lip(f^k(t, \cdot))\le \frac{1}{5}$, for $k\in \{1,5\}$ and $t\ge 0$;
\item $\|f^k(t, \cdot)\|_\infty \le  e^{-t}$ and $Lip(f^k(t, \cdot))\le \frac{1}{5}e^{-t}$, for $k\in \{2, 4\}$ and $t\ge 0$.
\end{itemize}
It is easy to verify that under the above assumptions, Theorem~\ref{THM2} is applicable.
\end{example}


\medskip{\bf Acknowledgements.}
 L.B. was partially supported by a CNPq-Brazil PQ fellowship under Grants No. 306484/2018-8 and 307633/2021-7. D. D. was supported in part by Croatian Science Foundation under the project IP-2019-04-1239 and by the University of Rijeka under the projects uniri-prirod-18-9 and uniri-prprirod-19-16.

\bibliographystyle{abbrv}

\end{document}